\documentclass[12pt,a4paper]{amsart}
\usepackage{amsmath,amsfonts,amssymb,amscd,amsthm,amsrefs}

\usepackage[shortlabels]{enumitem}

\usepackage{color, soul}
\definecolor{gr}{rgb}{0.7, 1, 0.7}
\definecolor{rr}{rgb}{1, 0.7, 0.7}

\usepackage{latexsym}

\usepackage{graphicx}

\theoremstyle{plain} 
\newtheorem{theorem}{Theorem}[section]

\newtheorem{thmx}{Theorem}

\newtheorem{lemma}[theorem]{Lemma}

\newtheorem{proposition}[theorem]{Proposition}

\theoremstyle{definition} 
\newtheorem{definition}[theorem]{Definition}
\theoremstyle{remark} 
\newtheorem{remark}[theorem]{Remark}

\renewcommand{\mathfrak}{\mathbf}

\renewcommand{\Im}{\,\mathrm{Im}\,}

\newcommand{\ignore}[1]{}

\newcommand{\bbC}{\mathbb{C}}

\newcommand{\bbN}{\mathbb{N}}
\newcommand{\bbR}{\mathbb{R}}
\newcommand{\bbZ}{\mathbb{Z}}
\newcommand{\bbQ}{\mathbb{Q}}

\newcommand{\bbD}{\mathbb{D}}

\newcommand{\bfu}{\mathbf{u}}

\newcommand{\bfw}{\mathbf{w}}

\newcommand{\bfv}{\mathbf{v}}

\newcommand{\tl}{\tilde}

\newcommand{\diam}{\operatorname{diam}}

\renewcommand{\Im}{\operatorname{Im}}
\renewcommand{\mod}{\operatorname{mod}}

\newcommand{\M}{\mathbb M}
\newcommand{\cc}{\mathfrak c}
\newcommand{\id}{\operatorname{id}}
\newcommand{\cycl}{\operatorname{cycl}}
\newcommand{\mbif}{\mu_{\mathrm{bif}}}
\newcommand{\Pern}{\mathrm{Per}_n}
\newcommand{\PernZ}{\mathrm{Per}_n/\bbZ_n}
\newcommand{\loc}{\mathrm{loc}}

\title[]{Equidistribution of critical points of the multipliers in the quadratic family}
\author{Tanya Firsova}
\thanks{Research of the first author was supported in part by NSF grant DMS-1505342}
\address{Kansas State University, Manhattan, KS, USA}
\email{tanyaf@math.ksu.edu}

\author{Igors Gorbovickis}
\address{Jacobs University, Bremen, Germany}
\email{i.gorbovickis@jacobs-university.de}

\subjclass[2010]{}
\keywords{}
\date{\today}

\begin{document}
\begin{abstract}
A parameter $c_0\in\bbC$ in the family of quadratic polynomials $f_c(z)=z^2+c$ is a \textit{critical point of a period $n$ multiplier}, if the map $f_{c_0}$ has a periodic orbit of period $n$, whose multiplier, viewed as a locally analytic function of $c$, has a vanishing derivative at $c=c_0$. We prove that all critical points of period $n$ multipliers 
equidistribute on the boundary of the Mandelbrot set, as $n\to\infty$.
\end{abstract}
\maketitle

\section{Introduction}

Consider the family of quadratic polynomials
$$
f_c(z)=z^2+c,\qquad c\in\bbC.
$$
We say that a parameter $c_0\in\bbC$ is a \textit{critical point of a period $n$ multiplier}, if the map $f_{c_0}$ has a periodic orbit of period $n$, whose multiplier, viewed as a locally analytic function of $c$, has a vanishing derivative at $c=c_0$. 

The study of these critical points is motivated by the following observation: the argument of quasiconformal surgery implies that appropriate inverse branches of the multipliers of periodic orbits, viewed as analytic functions of the parameter $c$, are Riemann mappings of the corresponding hyperbolic components of the Mandelbrot set~\cite{Milnor_hyper}. Possible existence of analytic extensions of these Riemann mappings to larger domains might allow to estimate the geometry of the hyperbolic components~\cites{Levin_2009,Levin_2011}. Critical values of the multipliers are the only obstructions for existence of these analytic extensions.

For each $n\in\bbN$, let $X_n$ be the set of all parameters $c\in\bbC$ that are critical points of a period $n$ multiplier (counted with multiplicities). 
Let $\M\subset\bbC$ denote the Mandelbrot set and let $\mbif$ be its equilibrium measure (or the bifurcation measure of the quadratic family $\{f_c\}$).

Our first main result is the following:

\begin{thmx}\label{main_theorem_1}
The sequence of probability measures
$$
\nu_n = \frac{1}{\# X_n}\sum_{x\in X_n}\delta_x
$$
converges to the equilibrium measure $\mbif$ in the weak sense of measures on $\bbC$, as $n\to\infty$.

\end{thmx}

In particular, Theorem~\ref{main_theorem_1} gives a positive answer to the question, stated in~\cite{Belova_Gorbovickis}.

We note that Theorem~\ref{main_theorem_1} is a partial case of a more general result that we prove in this paper. A precise statement of this more general result will be given in the next section (c.f.~ Theorem~\ref{main_equidistribution_theorem}).

\begin{figure}

\begin{center}
\begin{tabular}{c c}
 {\includegraphics[width=0.47\textwidth]{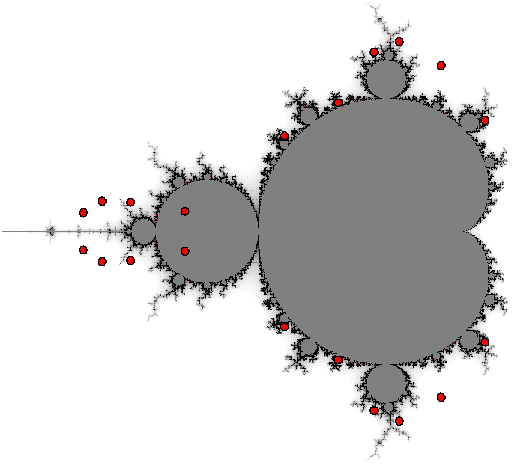}} & {\includegraphics[width=0.47\textwidth]{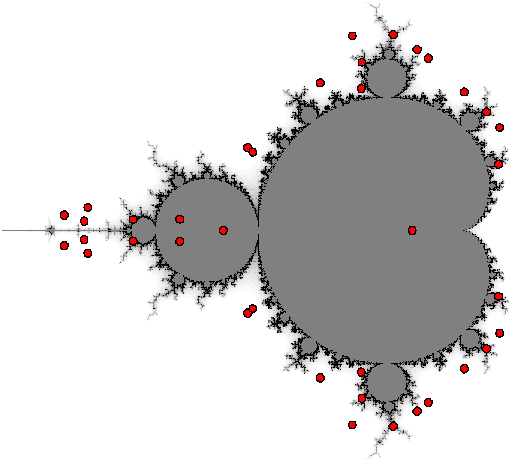}} \\
(a) & (b) \\
 {\includegraphics[width=0.47\textwidth]{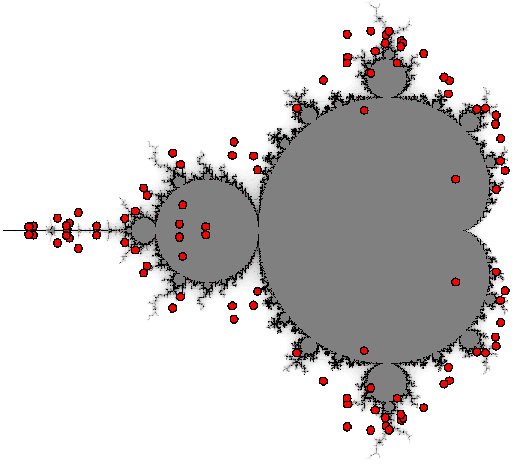}} & {\includegraphics[width=0.47\textwidth]{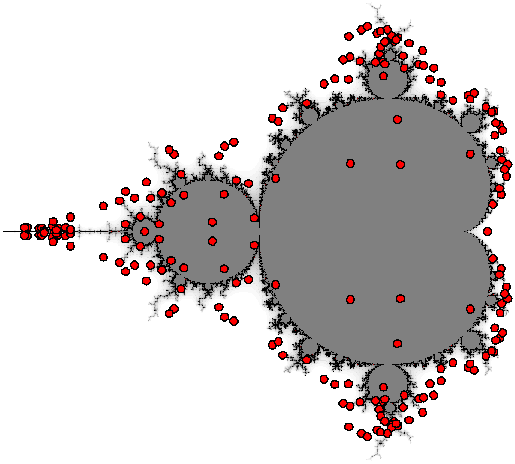}} \\
(c) & (d)
\end{tabular} 
\caption{Critical points of the multipliers: (a)~period~5; (b)~period~6; (c)~period~7; (d)~period~8. } 
\end{center}
\end{figure}

Equidistribution results in the parameter plane have attracted a lot of attention: starting from the results for critically periodic parameters of period $n$  \cite{Levin}, to more recent results for parameters with a prescribed multiplier ~\cites{Bassanelli_Berteloot,Buff_Gauthier} and Misiurewicz points~\cite{FRL, DF, Gauthier_Vigny_2019}. It is important to note that for all of the above mentioned classes of points in the quadratic family, their accumulation sets coincide with the support of $\mbif$, i.e., the boundary of the Mandelbrot set. The latter is not the case for critical points of the multipliers. In particular, we prove the following theorem.

\begin{thmx}\label{main_theorem_2}
For every $n_0\in\bbN$ and $c\in X_{n_0}\setminus\M$, there exists a sequence $\{c_n\}_{n=3}^\infty$, such that $c_n\in X_n$, for any $n\ge 3$, and 
$$
\lim_{n\to\infty} c_n=c.
$$
\end{thmx}

Theorem~\ref{main_theorem_2} can be loosely interpreted as follows: fix an arbitrary $n_0\in\bbN$. Then for any sufficiently large $n\in\bbN$, the set $X_n\setminus\M$ contains a ``distorted'' copy of $X_{n_0}\setminus\M$ as a subset. The larger is $n$, the closer is this copy to the original set $X_{n_0}\setminus\M$ (in Hausdorff metric).

We remark that the sequence of parameters $c_n$ in Theorem~\ref{main_theorem_2} starts from $n=3$, because the sets $X_1$ and $X_2$ are empty. We also note that according to the numerical computations in~\cite{Belova_Gorbovickis}, the sets $X_n\setminus\M$ are nonempty, for all $n=3,4,\dots,10$, so Theorem~\ref{main_theorem_2} implies that the sets $X_n\setminus\M$ are nonempty for all sufficiently large $n$.

Let $\mathcal X$ be the set of all accumulation points of the sets $X_n$, i.e.
$$
\mathcal X:= \bigcap_{k=3}^\infty \left(\,\overline{\bigcup_{n=k}^\infty X_n} \,\right).
$$
The above discussion suggests that this set might have nontrivial geometry. In particular, Theorem~\ref{main_theorem_2} implies the following inclusion:
$$
\overline{\bigcup_{n=3}^\infty (X_n\setminus\M)} \subset\mathcal X.
$$
We also know from~\cite{Belova_Gorbovickis} and~\cite{Reinke} that $0\in\mathcal X$. In an upcoming paper we will study some further geometric properties of the set $\mathcal X$.

The structure of the paper is as follows: in Section~\ref{Statemnt_of_results_sec} we give the necessary basic definitions and state our main results in a more precise form. In Section~\ref{Algebraic_sec} we describe the derivatives of the multipliers as algebraic functions (i.e., roots of polynomial equations). This allows us to explicitly compute potentials of the measures $\nu_n$. In Section~\ref{Proof_of_main_thm_sec} we give a proof of Theorem~\ref{main_equidistribution_theorem}, a more general version of Theorem~\ref{main_theorem_1}, modulo Lemma~\ref{Outside_M_convergence_lemma} that states convergence of potentials in the complement of the Mandelbrot set. A key tool in our proof is Lemma~\ref{Buff_Gauthier_lemma} that was proved by Buff and Gauthier in~\cite{Buff_Gauthier}. Sections~\ref{Ergodic_sec} and~\ref{Outside_M_sec} are devoted to the study of the multipliers of periodic orbits viewed as functions of the parameter $c\in\bbC\setminus\M$ in the complement of the Mandelbrot set. In this case it turns out to be more natural to study the degree $n$ roots of the multipliers, where $n$ is the period of the corresponding periodic orbits. In Section~\ref{Ergodic_sec} we use the Ergodic Theorem to prove that as $n\to\infty$, the roots of the multipliers of the majority of periodic orbits behave as twice the square root of the uniformizing coordinate of $\bbC\setminus\M$ onto $\bbC\setminus\overline{\bbD}$ (c.f., Theorem~\ref{main_frequency_theorem}). At the same time, we construct examples of sequences of periodic orbits, whose roots of the multipliers behave differently. This way we obtain a proof of Theorem~\ref{main_theorem_2}. Finally, in Section~\ref{Outside_M_sec} we use the results of Section~\ref{Ergodic_sec} to prove Lemma~\ref{Outside_M_convergence_lemma}, this way, completing the proof of Theorem~\ref{main_equidistribution_theorem}.

\section{Statement of results}\label{Statemnt_of_results_sec}

A point $z\in\bbC$ is a \textit{periodic point} of the polynomial $f_c$, if there exists a positive integer $n\in\bbN$, such that $f_c^{\circ n}(z)=z$. The smallest such $n$ is called the \textit{period} of the periodic point $z$.

Given $n$, let {\it the period $n$ curve} $\Pern \subset \mathbb{C}\times \mathbb{C}$ be the closure of the locus of points $(c,z)$ such that $z$ is a periodic point of $f_c$ of period $n$.
Observe that each pair $(c,z)\in \Pern$ determines a periodic orbit
\begin{equation*}
z=z_0\mapsto z_1 \mapsto \cdots \mapsto z_n=z_0
\end{equation*}
of either period $n$ or of a smaller period that divides $n$ (see ~\cite{Milnor_external_rays} for more details). 

Let $\mathbb{Z}_n$ denote the cyclic group of order $n$. This group acts on $\Pern$ by cyclicly permuting points of the same periodic orbits for each fixed value of $c$. 
Then the factor space $\PernZ$ consists of pairs $(c, \mathcal{O})$ such that $\mathcal{O}$ is a periodic orbit of $f_c$. 

Let $\tilde{\rho}_n: \Pern \to \mathbb{C}$ be the map defined by
$$
\tilde\rho_n\colon (c,z)\mapsto \frac{\partial f^{\circ n}_c}{\partial z}(z) = 2^n z_1 \cdots z_n.
$$
Observe that $\tilde\rho_n(c,z)$ is the multiplier of a periodic point $z$, whenever $z$ has period $n$. Otherwise, if a point $z$ has period $n/r$, where $r>1$ is some divisor of $n$, then $\tilde\rho_n(c,z)$ is the $r$-th power of the multiplier of $z$. 
Furthermore, if $z_1$ and $z_2$ belong to the same periodic orbit of $f_c$, then $\tilde\rho_n(c,z_1)=\tilde\rho_n(c,z_2)$, hence the map $\tilde\rho_n$ projects to a well defined map 
$$
\rho_n\colon \PernZ\to\mathbb C
$$
that assigns to each pair $(c,\mathcal O)$ the multiplier of the periodic orbit $\mathcal O$. Note that according to~\cite{Milnor_external_rays}, the space $\PernZ$ (as well as $\Pern$) has a structure of a smooth algebraic curve, and both $\rho_n$ and $\tilde\rho_n$ are proper algebraic maps.

It follows from the Implicit Function Theorem that if a point $(c,\mathcal O)\in\PernZ$ is such that the periodic orbit $\mathcal O$ has period less than $n$, then $\rho_n(c,\mathcal O)=1$.

\begin{definition}\label{parabolic_def}
	A point $(c,z)\in\Pern$ and its projection $(c,\mathcal O)\in\PernZ$ are called \textit{parabolic}, if $\rho_n(c,\mathcal O)=1$. A parabolic point $(c,z)\in\Pern$ and its projection $(c,\mathcal O)\in\PernZ$ are called \textit{primitive parabolic}, if the period of the point $z$ is equal to $n$. Otherwise, a parabolic point is called \textit{satellite}. 
The set of all primitive parabolic points of $\PernZ$ will be denoted by $\mathcal P_n\subset\PernZ$.
\end{definition}

\begin{remark}
Alternatively, one usually defines parabolic parameters $c$ as those, for which there exists a periodic orbit $\mathcal O$ of $f_c$ with a root of unity as its multiplier. Comparing this with Definition~\ref{parabolic_def}, we note that for every parabolic parameter $c\in\bbC$ (in the sense of the standard definition), there exist $n\in\bbN$ and  $z\in\bbC$, such that $(c,z)\in\Pern$, and the point $(c,z)\in\Pern$ is parabolic in the sense of Definition~\ref{parabolic_def}.
\end{remark}

It is well known that the coordinate $c$ can serve as a local chart on $\PernZ$ at all points $(c,\mathcal O)$ that are not primitive parabolic (c.f.~\cite{Milnor_external_rays}). Hence, outside of these points one can consider the derivative of the multiplier map $\rho_n$ with respect to $c$. Thus, we let the map 
$$
\sigma_n\colon(\PernZ)\setminus\mathcal P_n\to\bbC
$$
be defined by the relation
\begin{equation}\label{s_n_def_eq}
\sigma_n:= \frac{d}{dc}\rho_n.
\end{equation}
In particular, for every non-parabolic point $(c_0,\mathcal O)\in\PernZ$ that is the projection of a point $(c_0,z_0)\in\Pern$, one can define a locally analytic function $z(c)\in\bbC$, such that $z(c_0)=z_0$ and $(c,z(c))\in \Pern$, for all $c$ in a neighborhood of $c_0$. Then~(\ref{s_n_def_eq}) implies that
$$
\sigma_n(c_0,\mathcal O)=h'(c_0),\quad\text{where}\quad h(c)=\tilde\rho_n(c,z(c)).
$$

\begin{remark}\label{parabolic_growth_remark}
It follows from Lemma~4.5 from~\cite{Milnor_external_rays} that if $(c_0,\mathcal O_0)\in\mathcal P_n$ is a primitive parabolic point, then $|\sigma_n(c,\mathcal O)| \sim 1/\sqrt{|c-c_0|}$, as $(c,\mathcal O)\to(c_0,\mathcal O_0)$.
\end{remark}

\begin{definition}
For any $s\in\bbC$ and any $n\in\bbN$, let $Y_{s,n}\subset\PernZ$ be the set of all solutions of the equation $\sigma_n(c,\mathcal O)=s$. For any solution of this equation $(c,\mathcal O)\in Y_{s,n}$, let $\tl m_{s,n}(c,\mathcal O)$ be its multiplicity. 
Finally, let $X_{s,n}\subset\bbC$ be the projection of the set $Y_{s,n}$ onto the first coordinate, and for any $c\in X_{s,n}$, define $m_{s,n}(c)$ as
$$
m_{s,n}(c):= \sum_{\mathcal O\mid (c,\mathcal O)\in Y_{s,n}}\tl m_{s,n}(c,\mathcal O),
$$
where the summation goes over all periodic orbits $\mathcal O$, such that $(c,\mathcal O)\in Y_{s,n}$.
\end{definition}

We will show in Lemma~\ref{M_n_equals_deg_lemma} that for every $n\ge 3$, the number $\sum_{c\in X_{s,n}} m_{s,n}(c)$ is independent of the choice of $s\in\bbC$. Hence, for every $n\ge 3$ we define
\begin{equation}\label{M_n_def_eq}
M_n:=\sum_{c\in X_{s,n}} m_{s,n}(c).
\end{equation}

For $c\in\bbC$, let $\delta_c$ be the delta measure at the point $c$, and for every $n\ge 3$, and $s\in\bbC$, consider the probability measure
$$
\nu_{s,n}:= \frac{1}{M_n}\sum_{c\in X_{s,n}} m_{s,n}(c)\delta_c.
$$

For $c\in\bbC$, the Green's function $G_c\colon\bbC\to[0,+\infty)$ of the polynomial $f_c$ is given by
$$
G_c(z):=\lim_{n\to+\infty} \max\{2^{-n}\log|f_c^{\circ n}(z)|,\,\,0\},
$$
and the Green's function $G_\M\colon\bbC\to[0,+\infty)$ of the Mandelbrot set $\M$ satisfies 
$$
G_\M(c):= G_c(c)
$$
(see \cite{Douady_Hubbard} for details).
Finally, the bifurcation measure $\mbif$ is defined by
$$
\mbif:=\Delta G_\M,
$$
where $\Delta$ is the generalized Laplacian.

A more general version of our first main result (Theorem~\ref{main_theorem_1}) is the following:

\begin{theorem}\label{main_equidistribution_theorem}
For every sequence of complex numbers $\{s_n\}_{n\in\bbN}$, such that
\begin{equation}\label{lim_sup_condition}
\limsup_{n\to+\infty}\frac{1}{n}\log|s_n|\le \log 2,
\end{equation}
the sequence of measures $\{\nu_{s_n,n}\}_{n\in\bbN}$ converges to $\mbif$ in the weak sense of measures on $\bbC$.
\end{theorem}

\section{Derivatives of the multipliers as algebraic functions}\label{Algebraic_sec}

For every integer $n\ge 1$, let $\tilde{\mathcal P}_n\subset\bbC$ be the projection of the set of all primitive parabolic points $\mathcal P_n$ onto the first coordinate. That is,
$$
\tilde{\mathcal P}_n:=\{c\in\bbC\mid (c,\mathcal O)\in\mathcal P_n,\text{ for some periodic orbit }\mathcal O\}.
$$
\begin{remark}
For $n=1$ and $n=2$, the sets $\mathcal P_n$ and $\tilde{\mathcal P}_n$ are empty.
\end{remark}

Consider the functions $\tl S_n\colon(\bbC\setminus\tilde{\mathcal P}_n)\times\bbC\to\bbC$ defined by the formula
\begin{equation}\label{tilde_S_def_eq}
\tl S_n(c,s):=\prod_{\mathcal O\mid (c,\mathcal O)\in\PernZ} (s-\sigma_n(c,\mathcal O)),
\end{equation}
where the product is taken over all periodic orbits $\mathcal O$, such that $(c,\mathcal O)\in\PernZ$.
Also consider the polynomials $C_n\colon\bbC\to\bbC$ defined by the formula
$$
C_n(c):= \prod_{\tl c\in\tilde{\mathcal P}_n} (c-\tl c).
$$

One of the main results of this section is the following lemma:

\begin{lemma}\label{S_n_polynomial_lemma}
For any $n\geq 3$, define
$$
S_n(c,s):=C_n(c)\tl S_n(c,s).
$$
Then the function $S_n$ is a polynomial in $c$ and $s$, satisfying the property that $S_n(c,s)=0$ for a pair $(c,s)\in\bbC^2$, if and only if $s=\sigma_n(c,\mathcal O)$, for some $(c,\mathcal O)\in\PernZ$ (taking into the account multiplicities).

\end{lemma}
\begin{proof}
First, it follows from~(\ref{tilde_S_def_eq}) that for every $c\in\bbC\setminus\tilde{\mathcal P}_n$, the functions $\tl S_n(c,s)$ and hence the functions $S_n(c,s)$ are polynomials in $s$.

Next, we observe that $\tl S_n(c,s)$ is analytic in $c\in \bbC\setminus\tilde{\mathcal P}_n$. Furthermore, according to the Fatou-Shishikura inequality, for every $c_0\in\tilde{\mathcal P}_n$, there is exactly one parabolic point $(c_0,\mathcal O_0)\in\PernZ$ and the multiplicity of this point is equal to $2$ (i.e., when $c_0$ is perturbed to some nearby value $c\in\bbC$, the periodic orbit $\mathcal O_0$ splits into exactly two periodic orbits of period $n$). Now it follows from~(\ref{tilde_S_def_eq}) and Remark~\ref{parabolic_growth_remark} that for every $s\in\bbC$, the function $S(c)=\tl S_n(c,s)$ is meromorphic in $\bbC$ with simple poles at each point from the set $\tilde{\mathcal P}_n$.

Finally, we note that according to \cite{Buff_Gauthier}, $|\rho_n(c,\mathcal O)|\sim |4c|^{\frac n2}$ as $c\to \infty$, hence, for $n\ge 3$ we have $\sigma_n(c,\mathcal O)\to\infty$ as $c\to\infty$. Thus, for every $s\in\bbC$, the function $S(c)=\tl S_n(c,s)$ cannot have an essential singularity at infinity, hence is a rational function. Multiplication by $C_n(c)$ eliminates all simple poles at the points of the set $\tilde{\mathcal P}_n$, so the product $C_n(c)\tl S_n(c,s)$ extends to a polynomial in $\bbC^2$.

The second part of the lemma follows immediately from the construction of the polynomial $S_n$.
\end{proof}

For every $n\ge 3$, let $\deg_c S_n$ denote the highest degree of $S_n$ as a polynomial in $c$ with coefficients from the polynomial ring $\bbC[s]$.
\begin{lemma}\label{M_n_equals_deg_lemma}
For every $n\ge 3$, we have
$$
\deg_cS_n = M_n,
$$
where $M_n$ is defined in~(\ref{M_n_def_eq}) as the number of solutions $(c,\mathcal O)$ of the equation $\sigma_n(c,\mathcal O)=s$ (counted with multiplicity). In particular, this shows that $M_n$ is independent of $s\in\bbC$.
\end{lemma}
\begin{proof}
For any $s\in\bbC$, define $T_{n,s}(c):= S_n(c,s)$. Then $T_{n,s}$ is a polynomial in variable $c$.
It follows from Lemma~\ref{S_n_polynomial_lemma} that for any $s\in\bbC$, we have 
$$
\sum_{c\in X_{s,n}} m_{s,n}(c) = \deg T_{n,s}.
$$
We complete the proof by observing that $\deg T_{n,s} = \deg_cS_n$, for any $s\in\bbC$. The latter follows from the fact that according to~\cite{Buff_Gauthier}, $|\rho_n(c,\mathcal O)|\sim |4c|^{\frac n2}$, as $c\to \infty$, hence, for $n\ge 3$ we have $\sigma_n(c,\mathcal O)\to\infty$ as $c\to\infty$, which implies that the coefficient in front of the highest degree term in $c$ of the polynomial $S_n$ is a constant, independent of $s$.
\end{proof}

The next lemma summarizes the asymptotic behavior of the degrees of the polynomials $S_n$. 

\begin{lemma}\label{deg_c_S_n_lim_lemma}
The following limit holds:
\begin{equation*}
\lim_{n\to+\infty}2^{-n}(\deg_cS_n) =1.
\end{equation*}
\end{lemma}
\begin{proof}
For every $n\in\bbN$, let $\nu(n)$ be the number of periodic points of period $n$ for a generic quadratic polynomial $f_c$. The function $\nu(n)$ can be computed inductively by the formulas
\begin{equation}\label{nu_n_def_eq}
2^n=\sum_{r|n}\nu(r)\qquad \text{or}\qquad \nu(n)=\sum_{r|n}\mu(n/r)2^r,
\end{equation}
where the summation goes over all divisors $r\ge 1$ of $n$, and $\mu(n/r)\in\{\pm 1,0\}$ is the M\"obius function.

It is easy to see from the second formula that 
$$
\nu(n)\ge 2^n- \sum_{0\le j\le n/2} 2^j\ge 2^n-2^{\frac n2 +1}.
$$
On the other hand, since $\nu(n)\le 2^n$, it follows that
$$
\lim_{n\to\infty} 2^{-n}\nu(n)=1.
$$

It was shown in \cite{Belova_Gorbovickis} that $\deg_c S_n$ can be expressed in terms of the function $\nu(n)$ by the formula
$$
\deg_c S_n= \nu(n)-\frac{\nu(n)}{n}-\sum_{n=rp, p<n}\nu(p)\phi(r),
$$
where $\phi(r)$ is the number of positive integers that are smaller than $r$ and relatively prime with $r$. Since 
\begin{equation}\label{nu_phi_estimate_eq}
\sum_{n=rp, p<n}\nu(p)\phi(r) \le \sum_{n=rp, p<n} 2^{n/2}n \le \frac  n2\cdot 2^{n/2}n=2^{\frac n2-1}n^2, 
\end{equation}
it follows that 
$$
\lim_{n\to\infty} 2^{-n}\sum_{n=rp, p<n}\nu(p)\phi(r)=0,
$$
hence
$$
\lim_{n\to+\infty}2^{-n}(\deg_cS_n) =\lim_{n\to+\infty} 2^{-n}\left[\nu(n)-\frac{\nu(n)}{n} \right]=1.
$$
\end{proof}

\section{Proof of Theorem~\ref{main_equidistribution_theorem}}\label{Proof_of_main_thm_sec}

In this section we give a proof of Theorem~\ref{main_equidistribution_theorem} modulo the auxiliary lemmas stated below. The strategy of the proof follows the one from~\cite{Buff_Gauthier}.

Consider the subharmonic function $v\colon\bbC\to[0,+\infty)$ defined by
$$
v:= G_\M+\log 2.
$$

The following lemma was proven in \cite{Buff_Gauthier}.
\begin{lemma}[\cite{Buff_Gauthier}]\label{Buff_Gauthier_lemma}
Any subharmonic function $u\colon\bbC\to[-\infty,+\infty)$ which coincides with $v$ outside the Mandelbrot set $\M$, coincides with $v$ everywhere in $\bbC$.
\end{lemma}

Now for every $n\in\bbN$ and $s\in\bbC$, we define
\begin{equation}\label{u_sn_def_eq}
u_{s,n}(c):= (\deg_cS_n)^{-1}\log|S_n(c,s)|.
\end{equation}

Then it follows from Lemma~\ref{S_n_polynomial_lemma} and Lemma~\ref{M_n_equals_deg_lemma} that $u_{s,n}$ is a potential of the measure $\nu_{s,n}$, which means that
$$
\nu_{s,n}=\Delta u_{s,n}.
$$

The following lemma is crucial for the proof of Theorem~\ref{main_equidistribution_theorem}. 

\begin{lemma}\label{Outside_M_convergence_lemma}
For every sequence of complex numbers $\{s_n\}_{n\in\bbN}$, satisfying~(\ref{lim_sup_condition}), the sequence of subharmonic functions $\{u_{s_n,n}\}_{n\in\bbN}$ converges to $v$ in $L^1_\loc$ on the set $\bbC\setminus\M$.
\end{lemma}

The proof of Lemma~\ref{Outside_M_convergence_lemma} will occupy most of the remaining part of the paper. Note that in the special case of Theorem~\ref{main_theorem_1}, namely, the case when $s_n=0$ for every $n\in\bbN$, we can simplify the proof of Lemma~\ref{Outside_M_convergence_lemma} (see Remark~\ref{strategy_remark}).

Now we give a proof of Theorem~\ref{main_equidistribution_theorem} modulo Lemma~\ref{Outside_M_convergence_lemma}.

\begin{proof}[Proof of Theorem~\ref{main_equidistribution_theorem}]
It follows from Lemma~\ref{Outside_M_convergence_lemma} and Prokhorov's Theorem that the sequence of measures $\nu_n:= \nu_{{s_n},n}$ is sequentially relatively compact with respect to the weak convergence. Let $\nu$ be a probability measure that is a limit point of the sequence $\{\nu_n\}_{n\in\bbN}$. Then there exists a subsequence $\{n_k\}$, such that $\nu_{n_k}\to\nu$ in the weak sense of measures on $\bbC$, as $k\to\infty$, and the sequence $\{u_{s_{n_k},n_k}\}$ converges in $L^1_\loc$ to a subharmonic function $u$ on $\bbC$, satisfying 
$$
\Delta u=\nu.
$$
Furthermore, it follows from Lemma~\ref{Outside_M_convergence_lemma} that $u(c)=v(c)$, for all $c\in\bbC\setminus\M$, hence Lemma~\ref{Buff_Gauthier_lemma} implies that 
$$
u\equiv v\qquad\text{on}\quad\bbC.
$$ 
Thus, we conclude that the sequence of measures $\nu_n$ has only one limit point $\nu$, hence the sequence converges to $\nu$, where
$$
\nu=\Delta v= \Delta G_\M=\mbif.
$$
\end{proof}

\section{Roots of the multipliers and the ergodic theorem}\label{Ergodic_sec}
The aim of this section is to study the behavior of the multipliers of the maps $f_c$, when the parameter $c$ lies outside of the Mandelbrot set~$\M$ and the period $n$ of the periodic orbits increases to infinity. It turns out to be quite natural to look at the degree $n$ roots of the multipliers. Precise definitions are given in the following subsection.

\subsection{Roots of the multipliers outside the Mandelbrot set}

First, we summarize some well-known facts about the dynamics of $f_c$, when $c\in\bbC\setminus \M$. More details can be found in~\cite{Blanchard_Devaney_Keen}.

First, if $c\in\bbC\setminus \M$, then the Julia set $J_c$ of $f_c$ is a Cantor set and $0\not\in J_c$. The dynamics of $f_c$ on the Julia set $J_c$ is topologically conjugate to the Bernoulli shift on 2 symbols. Furthermore, the periodic points of $f_c$ move locally holomorphically with respect to the parameter $c$ when the latter varies outside of the Mandelbrot set $\mathbb M$, hence, by $\lambda$-Lemma \cites{Lyubich,ManeSadSullivan}, this holomorphic motion extends to a local holomorphic motion of the whole Julia set $J_c$. Since $\M$ is connected and simply connected, there exists a unique nontrivial monodromy loop in $\bbC\setminus \M$, namely, the loop that goes around $\M$ once. If we start with $c\in(1/4,+\infty)$ and make a loop around $\M$ (say, in the counterclockwise direction), then each point of $J_c$ comes back to its complex conjugateunder the above holomorphic motion. Going around this loop twice, brings every point of $J_c$ back to itself. This makes it natural to consider a degree 2 covering of $\bbC\setminus \M$.

More specifically, let 
$$
\phi_\M\colon\bbC\setminus\M\to\bbC\setminus\overline{\bbD}
$$
be the conformal diffeomorphism of $\bbC\setminus\M$ onto $\bbC\setminus\overline{\bbD}$ that sends the real ray $(1/4,+\infty)$ to $(1,+\infty)$. For $\lambda\in\bbC\setminus\overline{\bbD}$, set
\begin{equation}\label{cc_def_eq}
\cc(\lambda):=\phi_\M^{-1}(\lambda^2).
\end{equation}
Then $\cc\colon\bbC\setminus\overline{\bbD}\to\bbC\setminus\M$ is a covering map of degree~2. In addition to that, for every $\lambda\in \bbC\setminus\overline{\bbD}$, we have the following relation that will be useful in Section~\ref{Outside_M_sec}:
\begin{equation}\label{log_lambda_Green_eq}
\log|\lambda|=\frac{1}{2} G_\M(\cc(\lambda)).
\end{equation}

Let $\Omega:=\{0,1\}^\bbN$ be the space of all infinite binary sequences with the standard metric $d\colon\Omega\times\Omega\to\bbR^+$ defined as follows: if $\bfw_1=(\omega_0^1,\omega_1^1,\omega_2^1,\ldots)\in\Omega$ and $\bfw_2=(\omega_0^2,\omega_1^2,\omega_2^2,\ldots)\in\Omega$, then
$$
d(\bfw_1,\bfw_2)=2^{-k},
$$
where $k\in\{0\}\cup\bbN$ is the smallest index, for which $\omega_k^1\neq\omega_k^2$.
Let $\sigma\colon\Omega\to\Omega$ be the left shift. An element $\bfw=(\omega_0,\omega_1,\omega_2,\dots)\in\Omega$ is called an itinerary. There exists a uniquely defined one parameter family of maps
$$
\psi_\lambda\colon\Omega\to\bbC,\qquad \lambda\in\bbC\setminus\overline{\bbD}
$$ 
such that the following conditions hold simultaneously:
\begin{itemize}
\item for any $\lambda\in\bbC\setminus\overline{\bbD}$, the map $\psi_\lambda$ is a homeomorphism between $\Omega$ and $J_{\cc(\lambda)}$, conjugating $\sigma$ to $f_{\cc(\lambda)}$:
\begin{equation*} 
\psi_\lambda\circ\sigma=f_{\cc(\lambda)}\circ\psi_\lambda;
\end{equation*}
\item for each $\bfw\in\Omega$, the point $\psi_\lambda(\bfw)$ depends analytically on $\lambda\in\bbC\setminus\overline{\bbD}$;

\item for each $\lambda\in(1,+\infty)$ and $\bfw=(\omega_0,\omega_1,\omega_2,\dots)\in\Omega$, the point $\psi_\lambda(\bfw)$ is in the upper half-plane, if and only if $\omega_0=0$.
\end{itemize}

For further convenience in notation we define a function $\psi\colon\bbC\setminus \overline{\bbD}\times\Omega\to \bbC$ by the relation
$$
\psi(\lambda,\bfw)=\psi_\lambda(\bfw).
$$
Then, for each $\bfw\in\Omega$ we define the map $\psi_\bfw\colon\bbC\setminus\overline{\bbD}\to\bbC$ according to the relation
$$
\psi_\bfw(\lambda)=\psi(\lambda,\bfw).
$$

It follows from our construction that for each $\bfw=(\omega_0,\omega_1,\omega_2,\dots)\in\Omega$, the map $\psi_\bfw$ is holomorphic. Furthermore, 
\begin{equation}\label{psi_asymptotics_0_eq}
\psi_\bfw(\lambda)\sim i\lambda\text{ as }\lambda\to\infty,\quad\text{ when }\omega_0=0
\end{equation}
and 
\begin{equation}\label{psi_asymptotics_1_eq}
\psi_\bfw(\lambda)\sim -i\lambda\text{ as }\lambda\to\infty,\quad\text{ when }\omega_0=1.
\end{equation}

For any $n\in\bbN$ and $\bfw\in\Omega$, consider the product $2^n\prod_{k=0}^{n-1}\psi_{\sigma^k\bfw}(\lambda)$ as a function of $\lambda\in\bbC\setminus\overline{\bbD}$. It follows from~(\ref{psi_asymptotics_0_eq}) and~(\ref{psi_asymptotics_1_eq}) that this function has a local degree~$n$ at~$\infty$. Furthermore, since $0\not\in J_{\cc(\lambda)}$, for any $\lambda\in\bbC\setminus\overline\bbD$, this function never takes value zero, so any branch of its degree~$n$ root is a holomorphic function outside of the unit disk. We define a specific branch $g_{n,\bfw}\colon\bbC\setminus\overline{\bbD}\to\bbC$ of this root in the following way:

\begin{definition}
For any $z\in\bbC\setminus\{0\}$, let $\log z=\ln|z|+i\arg(z)$ be the branch of the logarithm, such that $\arg(z)\in(-\pi,\pi]$.
\end{definition}

\begin{definition}\label{g_n_def}
For any $n\in\bbN$ and $\bfw\in\Omega$, let the map $g_{n,\bfw}\colon\bbC\setminus\overline{\bbD}\to\bbC$ be the branch of the degree $n$ root $\left(2^n\prod_{k=0}^{n-1}\psi_{\sigma^k\bfw}(\lambda)\right)^{1/n}$ such that for any $\lambda\in(1,+\infty)$, the following holds:
\begin{equation*} 
g_{n,\bfw}(\lambda)= 2\exp\left(\frac{1}{n}\sum_{k=0}^{n-1}\log(\psi_\lambda(\sigma^k\bfw))\right).
\end{equation*}
\end{definition}

For notational convenience we also introduce the functions $g_n\colon\bbC\setminus\overline{\bbD}\times\Omega\to\bbC$ defined by
$$
g_n(\lambda,\bfw):=g_{n,\bfw}(\lambda).
$$ 
We note that the above relations define the map $g_n$ for all $\lambda\in\bbC\setminus\overline{\bbD}$ by analytic continuation in the variable $\lambda$, and since for any $\lambda\in(1,+\infty)$, the Julia set $J_{\cc(\lambda)}$ has no points on the real line, hence, no points on the ray $(-\infty,0]$, it follows that the map $g_n$ is continuous in $\bfw$.

For periodic itineraries, it is convenient to introduce the following functions:
\begin{definition}\label{periodic_g_w_def}
For each periodic itinerary $\bfw\in\Omega$, we let $n=n(\bfw)$ be the period of $\bfw$ and define the maps $\rho_\bfw, g_\bfw\colon\bbC\setminus\overline{\bbD}\to\bbC$ according to the formulas
\begin{equation*} 
\rho_\bfw(\lambda):=2^n\prod_{k=0}^{n-1}\psi_{\sigma^k\bfw}(\lambda)\qquad\text{and}
\end{equation*}
$$
g_\bfw(\lambda):= g_n(\lambda,\bfw).
$$
\end{definition}

Observe that if $\bfw\in\Omega$ is a periodic itinerary of period $n$, then $\rho_\bfw(\lambda)$ is the multiplier of the corresponding periodic orbit 
$$
\{\psi_\bfw(\lambda),\psi_{\sigma\bfw}(\lambda),\dots,\psi_{\sigma^{n-1}\bfw}(\lambda)\}\subset\bbC
$$
of the quadratic polynomial $f_{\cc(\lambda)}$, and $\rho_\bfw(\lambda)=[g_\bfw(\lambda)]^n$.

For each $n\in\bbN$, let $\Omega_n\subset\Omega$ be the finite set of all itineraries of period $n$. \begin{remark}\label{strategy_remark} 

If the sequence $\{s_n\}_{n\in\bbN}$ in Lemma~\ref{Outside_M_convergence_lemma} is identically zero, we can simplify the proof of Lemma \ref{Outside_M_convergence_lemma}. Below we give a sketch of the proof in this case. According to~(\ref{u_sn_def_eq}) and Lemma~\ref{S_n_polynomial_lemma}, we have 
$$
u_{0,n}(c)= F_n(c)+G_n(c), $$
where
$$
F_n(c) = (\deg_cS_n)^{-1}\log |C_n(c)|\qquad\text{and}
$$
$$
G_n(c) = (\deg_cS_n)^{-1}\log|\tilde S_n(c,0)|.
$$
One can show that for any $c\in\bbC\setminus\M$, we have
$$
\lim_{n\to\infty} F_n(c) = \log|\lambda|,
$$
where $\lambda$ is such that $\cc(\lambda)=c$ (see Propositions~\ref{P_n_lim_prop} and~\ref{lim_log_C_n_prop} combined with Lemma~\ref{deg_c_S_n_lim_lemma}). At the same time,
\begin{multline*}
G_{n}(c) =(\deg_cS_n)^{-1}\sum_{\mathcal O\mid (c,\mathcal O)\in\PernZ} \log|\sigma_n(c,\mathcal O)| \\
= \frac{(\deg_cS_n)^{-1}}{n} \sum_{\bfw\in \Omega_{n}} \log\left|\frac{d}{dc}[(g_{\bfw}(\lambda))^n]\right| \\
=\frac{(\deg_cS_n)^{-1}}{n} \sum_{\bfw\in \Omega_{n}} \left(\log n+(n-1)\log|g_\bfw(\lambda)| + \log|g_{\bfw}'(\lambda)|+\log\left|\frac{d\lambda}{dc}\right|\right).
\end{multline*}
Since the family of maps $\{g_\bfw\mid \bfw\in\cup_{n=1}^\infty\Omega_n\}$ is normal (see Proposition~\ref{normality_prop}), it follows that 
$$
\frac{1}{n}\log|g_{\bfw_n}'(\lambda)| \to 0\qquad \text{in }L^1_\loc\text{ on }\lambda\in\bbC\setminus\overline{\bbD},\text{ as }n\to\infty,
$$
\noindent uniformly over all $\bfw_n\in\Omega_n$.

Thus, combining these estimates with Lemma~\ref{deg_c_S_n_lim_lemma}, we obtain
$$
G_n(c) \to \frac{1}{2^n} \sum_{\bfw\in \Omega_{n}} \log|g_\bfw(\lambda)| \qquad \text{in }L^1_\loc\text{ on }c\in\bbC\setminus\M,\text{ as }n\to\infty.
$$

The latter expression can be shown to converge to $\log 2+\log|\lambda|$ by standard methods. This completes the proof of Lemma~\ref{Outside_M_convergence_lemma} in our special case.

In the general case of an arbitrary sequence $\{s_n\}_{n\in\bbN}$ satisfying~(\ref{lim_sup_condition}), the potentials $u_{s_n,n}$ can be represented as $u_{0,n}$ plus an additional term:
$$
u_{s_n,n}(c)= u_{0,n}(c) + (\deg_cS_n)^{-1} \sum_{\mathcal O\mid (c,\mathcal O)\in\PernZ} \log|1-s_n/\sigma_n(c,\mathcal O)|. 
$$
The main difficulty of the proof of Lemma~\ref{Outside_M_convergence_lemma} in the general case consists of estimating this additional term. Our strategy is, using the ergodic theorem to show that even though the maps $g_\bfw$ can have complicated behavior, the majority of them is close to $2\lambda$, as $n\to\infty$. We carry out the estimates for those ``tame'' maps, while showing that the remaining ones do not affect the limiting potential.

\end{remark}

For any compact subset $K\subset\bbC\setminus\overline{\bbD}$, let $\|\cdot\|_K$ denote the $C^0$-norm on the space of continuous functions defined on $K$.

One of the main goals of this section is to prove the following theorem that is informally stated in the end of Remark~\ref{strategy_remark}:

\begin{theorem}\label{main_frequency_theorem}

For any $\delta>0$ and a compact subset $K\subset\bbC\setminus\overline{\bbD}$, the following holds:
$$
\lim_{n\to\infty} \frac{\#\{\bfw\in\Omega_n\colon \|g_\bfw-2\cdot\id\|_K<\delta\}}{\#\Omega_n} = 1.
$$

\end{theorem}
We postpone the proof of Theorem~\ref{main_frequency_theorem} until Subsection~\ref{uniform_measure_subsec}.

\subsection{Continuity properties of the maps $g_n$}

In this subsection we prove a certain continuity property of the family of maps $\{g_n\mid n\in\bbN\}$. The property is much weaker than equicontinuity and can roughly be stated as follows: for any two itineraries $\bfw_1$ and $\bfw_2$, whose first $n$ digits match, one can guarantee the difference $|g_{n+k}(\lambda,\bfw_1)-g_{n+k}(\lambda,\bfw_2)|$ to be arbitrarily small for an arbitrarily large $k\in\bbN$ by requiring $n$ to be sufficiently large. The precise statement is given in the following lemma:

\begin{lemma}
\label{neighborhood_lemma}
	For any compact set $K\subset\bbC\setminus\overline{\bbD}$ and for any $k\in\bbN\cup\{0\}$ and  $\varepsilon>0$, there exists $N_0=N_0(K,k,\varepsilon)>0$, such that for every $n\ge N_0$ and $\bfw_1,\bfw_2\in\Omega$ with the property that $d(\bfw_1,\bfw_2)\le 2^{-n}$, the inequality
	$$
	|g_{n+k}(\lambda,\bfw_1)-g_{n+k}(\lambda,\bfw_2)|<\varepsilon
	$$
	holds for all $\lambda\in K$.
\end{lemma}

The proof of Lemma~\ref{neighborhood_lemma} is based on the idea that can loosely be stated as follows: if finite orbits of the points $\psi_\lambda(\bfw_1)$ and $\psi_\lambda(\bfw_2)$ under dynamics of the map $f_{\cc(\lambda)}$ shadow each other for a long time and then spend some fixed time apart, then the averages (geometric means) of the points of these finite orbits stay close to each other.

We need a few propositions before we can give a proof of Lemma~\ref{neighborhood_lemma}.

\begin{proposition}\label{g_k_m_prop}
	For any $\bfw\in\Omega$, $\lambda\in (1,+\infty)$ and for any $n,m\in\bbN$, the following relation holds:
	\begin{equation*}
	g_{n+m}(\lambda,\bfw)= [g_{n}(\lambda,\bfw)]^\frac{n}{n+m} \cdot [g_{m}(\lambda,\sigma^n\bfw)]^\frac{m}{n+m},
	\end{equation*}
	where the branches of the power maps $z\mapsto z^\frac{n}{n+m}$ and $z\mapsto z^\frac{m}{n+m}$ are chosen so that they are continuous on $\bbC\setminus (-\infty,0)$ and send the ray $(0,+\infty)$ to itself and the closed upper halfplane to the closed upper halfplane.
\end{proposition}
\begin{proof}
	The proposition follows from Definition~\ref{g_n_def} 
	 by a direct computation.
\end{proof}

\begin{proposition}\label{normality_prop}
	The family of holomorphic maps 
	$$
	\{g_{n,\bfw}\colon\bbC\setminus\overline{\bbD}\to\bbC\mid \bfw\in\Omega,n\in\bbN\}
	$$
	is normal.
\end{proposition}
\begin{proof}
	Since the Julia sets $J_{\cc(\lambda)}$ are compact and move holomorphically with respect to $\lambda$, then the considered family of maps is locally bounded, hence normal.
\end{proof}

As a corollary from these two propositions, we prove the following:
\begin{proposition}\label{g_n_difference_convergence_prop}
	For any $m\in\bbN$, the sequence of functions $h_{n,m}\colon(\bbC\setminus\overline\bbD)\times\Omega\to\bbC$, defined by 
	$$
	h_{n,m}(\lambda,\bfw)=g_{n+m,\bfw}(\lambda)-g_{n,\bfw}(\lambda),
	$$
	converges to zero uniformly on compact subsets of $(\bbC\setminus\overline\bbD)\times \Omega$, as $n\to\infty$. 
\end{proposition}
\begin{proof}
	We fix a number $m\in\bbN$ throughout the entire proof. Now, for any $\lambda\in(1,+\infty)$, there exists a real number $r>0$, such that the Julia set $J_{\cc(\lambda)}$ is contained in the round annulus centered at zero, with outer radius $r$ and inner radius $1/r$. According to Definition~\ref{g_n_def}, this implies that for any $\lambda\in(1,+\infty)$, 
	$$
	\lim_{n\to\infty} [g_{m}(\lambda,\sigma^n\bfw)]^\frac{m}{n+m} = 1,
	$$
	and the convergence is uniform in $\bfw\in\Omega$.
	
	The last limit together with Proposition~\ref{g_k_m_prop} and uniform boundedness of the functions $g_n$ on $\{\lambda\}\times\Omega$ implies that for any $\lambda\in(1,+\infty)$, we have
	\begin{equation}\label{hn_w_uniform_equation}
	\lim_{n\to\infty} h_{n,m}(\lambda,\bfw)=0,
	\end{equation}
	and the convergence is uniform in $\bfw\in\Omega$.

	Now assume that there is no uniform convergence of the functions $h_{n,m}$ to zero on compact subsets of $(\bbC\setminus\overline\bbD)\times \Omega$, as $n\to\infty$. This implies that there exists $\varepsilon>0$, a compact set $K\subset\bbC\setminus\overline\bbD$ and a sequence of triples $\{(n_k,\lambda_k,\bfw_k)\in\bbN\times K\times \Omega \mid k\in\bbN\}$, such that $n_j>n_k$, whenever $j>k$ and
	\begin{equation}\label{hnk_inequality}
	h_{n_k, m}(\lambda_k,\bfw_k)>\varepsilon,\qquad\text{for all } k\in\bbN.
	\end{equation}
	Consider the sequence of maps $h_k\colon\bbC\setminus\overline\bbD\to\bbC$ defined by $h_{k}(\lambda)=h_{n_k,m}(\lambda,\bfw_k)$. It follows from Proposition~\ref{normality_prop} that this sequence is normal. Let $h\colon \bbC\setminus\overline\bbD\to\bbC$ be its arbitrary limit point. Inequality~(\ref{hnk_inequality}) implies that $h\not\equiv 0$ on $K$, and since $h$ is a holomorphic map, this implies that there exists $\lambda\in (1,+\infty)$, such that $h(\lambda)\neq 0$. The latter contradicts to the uniform convergence in $\bfw\in\Omega$, established in~(\ref{hn_w_uniform_equation}). 
\end{proof}

\begin{proposition}\label{root_average_prop}
For any real numbers $C,\delta\in\bbR$, such that $C>1$ and $0<\delta<1/C$, the following holds: for any $n\in\bbN$ and any points $x_1,\dots,x_n, y_1,\dots,y_n\in\bbC$, satisfying $1/C\le|x_j|\le C$ and $|y_j-x_j|\le \delta$ for each $j=1,\dots,n$, we have
$$
|\sqrt[n]{y_1\dots y_n}-\sqrt[n]{x_1\dots x_n}|\le C^2\delta,
$$ 
for any branch of the first root and an appropriately chosen branch of the second root.
\end{proposition}
\begin{proof}
$$
|\sqrt[n]{y_1\dots y_n}-\sqrt[n]{x_1\dots x_n}| = |\sqrt[n]{x_1\dots x_n}|\cdot \left|\sqrt[n]{\frac{y_1\dots y_n}{x_1\dots x_n}}-1\right| \le 
$$
$$
C\left(\frac{(1/C)+\delta}{1/C}-1\right) = C^2\delta.
$$ 
\end{proof}

\begin{proof}[Proof of Lemma~\ref{neighborhood_lemma}]

	Since the function $\psi$ is continuous on the compact set $K\times\Omega$,  
	there exists $N_1=N_1(K)>0$, such that for any $\lambda\in K$ and $\bfw_1,\bfw_2\in\Omega$ with the property that $d(\bfw_1,\bfw_2)\le 2^{-N_1}$, we have $|\psi(\lambda,\bfw_1)-\psi(\lambda,\bfw_2)|<\varepsilon/(4C^2)$, where 
	$$
	C=\max_{(\lambda,\bfw)\in K\times\Omega}(\max\{|\psi(\lambda,\bfw)|, |\psi(\lambda,\bfw)|^{-1}\}).
	$$
	(The maximum exists since the set $K\times\Omega$ is compact.) 
	
	Now it follows from Definition~\ref{g_n_def} and Proposition~\ref{root_average_prop} that if $n>N_1$ and $d(\bfw_1,\bfw_2)\le 2^{-n}$, then 
	\begin{equation}\label{middle_term_ineq}
	|g_{n-N_1}(\lambda,\bfw_1)-g_{n-N_1}(\lambda,\bfw_2)|<\varepsilon/2.
	\end{equation}
	
	Now we observe that if $n>N_1$, then
	$$
	g_{n+k}(\lambda,\bfw_1)-g_{n+k}(\lambda,\bfw_2)= A_n+B_n+C_n,
	$$
	where
	$$
	A_n=g_{n+k}(\lambda,\bfw_1)-g_{n-N_1}(\lambda,\bfw_1),
	$$
	$$
	B_n=g_{n-N_1}(\lambda,\bfw_1)-g_{n-N_1}(\lambda,\bfw_2),
	$$
	$$
	C_n=g_{n-N_1}(\lambda,\bfw_2)-g_{n+k}(\lambda,\bfw_2).
	$$
	It follows from~(\ref{middle_term_ineq}) that $|B_n|<\varepsilon/2$, while Proposition~\ref{g_n_difference_convergence_prop} implies existence of a positive integer $N_0>N_1$, such that for all $n>N_0$, we have $|A_n|<\varepsilon/4$ and $|C_n|<\varepsilon/4$. This completes the proof of the lemma.
\end{proof}

\subsection{Special accumulation points of the maps $g_\bfw$.}

{In this subsection we apply Lemma~\ref{neighborhood_lemma} to give a proof of Theorem B.}

For any $k\in\bbN$ and for any finite sequence of $k$ elements $\bfv\in\{0,1\}^k$, let $\langle\bfv\rangle\in\Omega$ denote the infinite sequence obtained from $\bfv$ by repeating it infinitely many times. In particular, the sequence $\langle\bfv\rangle$ is periodic with period dividing~$k$.

\begin{proposition}\label{tail_choice_prop}
For any $n\in\bbN$ and for any two distinct finite sequences $\bfv,\bfu\in\{0,1\}^n$ that differ only in the last $n$-th digit, at least one of the sequences $\langle\bfv\rangle$ and $\langle\bfu\rangle$ has period $n$.
\end{proposition}
\begin{proof}
Let $k$ and $m$ be the periods of $\langle\bfv\rangle$ and $\langle\bfu\rangle$ respectively. Let $\langle\bfv\rangle= (v_1,v_2,\dots)$ and $\langle\bfu\rangle= (u_1,u_2,\dots)$. Assume, both $k<n$ and $m<n$. If $k=m$, then $u_{n-k}=u_n\neq v_n=v_{n-k}$ which is a contradiction to the assumption of the proposition. Now if $k\neq m$, then
$$
v_{n+k}=v_n\neq u_n=u_{n-m}.
$$
Finally, since none of the numbers $n+k$, $n-m$ and $n+k-m$ are divisible by $n$, we obtain that
$$
v_{n+k-m}=u_{n+k-m}=u_{n+k}=v_{n+k}\qquad\text{and}
$$
$$
u_{n+k-m}=v_{n+k-m}= v_{n-m}=u_{n-m}.
$$
Together with the previous inequality, this implies that $v_{n+k-m}\neq u_{n+k-m}$, which again contradicts to the assumption of the proposition. 
\end{proof}

\begin{theorem}\label{accumulation_of_roots_theorem}
For any periodic itinerary $\bfw\in\Omega$, there exists a sequence of periodic itineraries $\bfw_1,\bfw_2,\bfw_3,\ldots\in\Omega$, such that $\bfw_j\in\Omega_j$, for any $j\in\bbN$ and the sequence of maps $g_{\bfw_1},g_{\bfw_2},g_{\bfw_3},\ldots$ converges to the map $g_\bfw$ uniformly on compact subsets of $\bbC\setminus\overline{\bbD}$.
\end{theorem}
\begin{proof}
Let $n$ be the period of the itinerary $\bfw$ and let $\bfv\in\{0,1\}^n$ be a finite sequence, such that $\bfw=\langle\bfv\rangle$. Any positive integer $j\in\bbN$ can be represented as
$$
j=s_jn+r_j,
$$
where $s_j,r_j\in\bbZ$ and $1\le r_j\le n$. We define $\bfv_j\in\{0,1\}^j$ as a finite sequence obtained by taking $s_j$ copies of $\bfv$ followed by some $r_j$ digits so that the itinerary $\langle\bfv_j\rangle$ has period $j$. The latter is possible due to Proposition~\ref{tail_choice_prop}. Finally, for any $j\in\bbN$, we define $\bfw_j:=\langle\bfv_j\rangle$. Then it follows that for any compact $K\subset\bbC\setminus\overline{\bbD}$ we have
$$
\|g_{\bfw_j}-g_\bfw\|_K = \|g_{j,\bfw_j}-g_{s_jn,\bfw}\|_K\le \|g_{j,\bfw_j}-g_{s_jn,\bfw_j}\|_K + \|g_{s_jn,\bfw_j}-g_{s_jn,\bfw}\|_K.
$$
According to Proposition~\ref{g_n_difference_convergence_prop} and Lemma~\ref{neighborhood_lemma}, the last two terms in the above inequality converge to zero as $j\to\infty$. This completes the proof of the theorem.
\end{proof}

\begin{proof}[Proof of Theorem~\ref{main_theorem_2}]
If $c\in X_{n_0}\setminus \M$, for some $n_0\in\bbN$, then there exists $\lambda\in\bbC\setminus\overline{\bbD}$, such that $c=\cc(\lambda)$ and $\lambda$ is an isolated critical point of the map $g_\bfw$, for some $\bfw\in\Omega_{n_0}$. Let $\{\bfw_n\}_{n=1}^\infty$ be the sequence of periodic itineraries from Theorem~\ref{accumulation_of_roots_theorem}. Then the sequence of maps $\{g_{\bfw_n}\}_{n=1}^\infty$ converges to $g_\bfw$ on compact subsets of $\bbC\setminus\overline{\bbD}$. This implies that the sequence of derivatives $\{g_{\bfw_n}'\}_{n=1}^\infty$ converges to $g_\bfw'$ on compact subsets of $\bbC\setminus\overline{\bbD}$. Since $\lambda$ is an isolated zero of the map $g_\bfw'$, it follows that there exists a sequence of points $\lambda_n\in\bbC\setminus\overline{\bbD}$, such that 
$$
\lim_{n\to\infty}\lambda_n = \lambda,
$$
and $g_{\bfw_n}'(\lambda_n)=0$, for every $n\ge 3$. We complete the proof of Theorem~\ref{main_theorem_2} by setting $c_n:=\cc(\lambda_n)$, for every $n\ge 3$.
\end{proof}

\subsection{Ergodic theorem}

For any Borel probability measure $\mu$ on $\Omega$, we define an analytic map $\overline{\psi}_\mu\colon\bbC\setminus\overline{\bbD}\to\bbC$ in the following way: first, for any $\lambda\in(1,+\infty)$, we set
\begin{multline*} 
\overline{\psi}_\mu(\lambda):= \exp\left(\int_{\Omega}\log(2\psi_\lambda)\,d\mu\right) =\\  \exp\left(\int_{\Omega}\log|2\psi_\lambda|\,d\mu+i\int_\Omega \arg(2\psi_\lambda)\,d\mu\right),
\end{multline*}
where $\arg(2\psi_\lambda)\in(-\pi,\pi]$. (The integrals are well defined, since for each $\lambda\in(1,+\infty)$, the function $\psi_\lambda$ is continuous, hence $\mu$-measurable, and bounded away from zero and infinity.) Then, we observe that for any $\bfw\in\Omega$ and for any closed loop $\gamma\colon S^1\to\bbC\setminus\overline{\bbD}$, going once around $\bbD$ in the counterclockwise direction, the loop $\psi_\bfw(\gamma)$ has winding number~1 around the origin, hence the integral $\int_{\Omega}\log(2\psi_\lambda)\,d\mu$ increases by $2\pi i$ after analytic continuation along such a closed loop $\gamma$. The later implies that the function $\overline{\psi}_\mu$ defined 
on the ray $\lambda\in(1,+\infty)$, admits analytic continuation to the entire domain $\bbC\setminus\overline{\bbD}$.

\begin{remark}
One can informally think of the map $\overline{\psi}_\mu$ as a ``complexified version'' of the Lyapunov characteristic exponent $\int_{\Omega}\log|2\psi_\lambda|\,d\mu$ of the map $f_{\cc(\lambda)}$ in the complement of the Mandelbrot set (see Chapter~10 of \cite{PrzU} for a detailed discussion of Lyapunov characteristic exponents of analytic maps).  Analyticity of the map $\overline{\psi}_\mu$ turns out to be quite handy in the further discussion.

\end{remark}

\begin{theorem}[\textbf{Ergodic Theorem}]\label{ergodic_theorem}
For any ergodic Borel probability measure $\mu$ on $\Omega$, and 
for $\mu$-a.e. $\bfw\in\Omega$, the sequence of maps $\{g_{n,\bfw}\}_{n\in\bbN}$ converges to $\overline{\psi}_\mu$ on compact subsets of $\bbC\setminus\overline{\bbD}$, as $n\to\infty$.
\end{theorem}

\begin{proof}
For any rational $\lambda\in\bbQ\cap (1,+\infty)$ and for any $\bfw\in\Omega$, it follows from Definition~\ref{g_n_def} that
$$
\log [g_n(\lambda,\bfw)]= \frac{1}{n}\sum_{k=0}^{n-1}\log(2\psi_\lambda(\sigma^k\bfw)),
$$
hence according to Birkhoff's Ergodic Theorem, there exists a set $\Omega_\lambda\subset\Omega$ with $\mu(\Omega_\lambda)=1$, such that for any $\bfw\in\Omega_\lambda$, we have
$$
\lim_{n\to\infty} \log [g_n(\lambda,\bfw)]= \int_{\Omega}\log(2\psi_\lambda)\,d\mu, 
$$
which in turn implies
\begin{equation}\label{ergodic_limit_eq}
\lim_{n\to\infty} g_n(\lambda,\bfw)= \overline{\psi}_\mu(\lambda).
\end{equation}

Define 
$$
\tl\Omega:= \cap_{\lambda\in\bbQ\cap(1,+\infty)}\Omega_\lambda.
$$
Since this is a countable intersection of full measure sets, we conclude that $\mu(\tl\Omega)=1$, and according to the construction of the set $\tl\Omega$, identity~(\ref{ergodic_limit_eq}) holds for all $\lambda\in\bbQ\cap (1,+\infty)$ and $\bfw\in\tl\Omega$.

On the other hand, Proposition~\ref{normality_prop} implies that for any itinerary $\bfw\in\tl\Omega$, the family of holomorphic maps $\mathcal F_\bfw=\{g_{n,\bfw}\mid n\in\bbN\}$ is normal. Then any limit point of this family is a holomorphic map on $\bbC\setminus\overline{\bbD}$ that must coincide with $\overline{\psi}_\mu$ at all rational points of the ray $(1,+\infty)$. The latter implies that the limit point of the family $\mathcal F_\bfw$ is unique and is equal to $\overline{\psi}_\mu$ on the entire domain $\bbC\setminus\overline{\bbD}$.
\end{proof}

\subsection{The case of uniform measure}\label{uniform_measure_subsec}
In this subsection we prove Theorem~\ref{main_frequency_theorem} by applying the Ergodic Theorem~\ref{ergodic_theorem} in the case of the uniform ergodic measure on~$\Omega$. 

Let $\mu_0$ be the uniform Borel measure on $\Omega$ defined on the cylinders
$$
[(i_1,s_1),\dots,(i_k,s_k)]:= \{\bfw=(\omega_j)_{j=0}^\infty\in\Omega\mid \omega_{i_1}=s_1,\ldots, \omega_{i_k}=s_k\}
$$
by the relation
$$
\mu_0([(i_1,s_1),\dots,(i_k,s_k)])=2^{-k}.
$$
It is well known that this measure $\mu_0$ is ergodic for the left shift $\sigma$ on~$\Omega$.

\begin{lemma}\label{space_average_lemma}
For any $\lambda\in\bbC\setminus\overline{\bbD}$, we have
$\overline{\psi}_{\mu_0}(\lambda)=2\lambda$.
\end{lemma}
\begin{proof}

First, we observe that
\begin{equation}\label{Lyap_exp_formula}
\int_{\Omega}\log|2\psi_\lambda|\,d\mu_0= \log|2\lambda|,
\end{equation}
which follows directly from a more general formula 
$$
L(f)=\log d + \sum_{\{c\in\bbC\colon f'(c)=0\}} G_f(c),
$$
where $f$ is a polynomial of degree $d>1$, $L(f)$ is its Lyapunov exponent and $G_f$ is the Green's function of the filled Julia set of $f$ (see \cite{Manning, Prz}).
For the sake of self-containment, in the next paragraph we give a short proof of~(\ref{Lyap_exp_formula}) in our special case.

For each $\lambda\in\bbC\setminus\overline{\bbD}$, the map $\log|2\psi_\lambda|$ can be approximated by step functions $\chi_{n,\lambda}\colon\Omega\to\bbR$ that are constant on cylinders of length $n$ and defined as follows: if $\bfw=(\omega_0,\omega_1,\dots)\in\Omega$ and $\bfv= (\omega_0,\dots,\omega_{n-1})\in\{0,1\}^n$ is a finite sequence of the first $n$ digits of $\bfw$, then
$$
\chi_{n,\lambda}(\bfw):= \log|2\psi_\lambda(\langle \bfv\rangle)|.
$$
Then it is clear that for each $\lambda\in\bbC\setminus\overline{\bbD}$, the maps $\chi_{n,\lambda}$ are uniformly bounded and converge to $\log|2\psi_\lambda|$ pointwise, so according to the Lebesgue convergence theorem, we have
\begin{multline*}
\int_{\Omega}\log|2\psi_\lambda|\,d\mu_0=\lim_{n\to\infty}\frac{1}{2^n}\sum_{\bfv\in\{0,1\}^n}\log|2\psi_\lambda(\langle \bfv\rangle)| \\
= \log 2 + \lim_{n\to\infty}\frac{1}{2^{n}}\log|f_c^{\circ n}(0)| = \log 2 + \lim_{n\to\infty}\frac{1}{2^{n}}\log|f_c^{\circ (n-1)}(c)| \\
= \log 2+ \frac{1}{2}G_\M(c) = \log|2\lambda|.
\end{multline*}

Now, using~(\ref{Lyap_exp_formula}), we conclude that
\begin{equation*}
|\overline{\psi}_{\mu_0}(\lambda)|=   \exp\left(\int_{\Omega}\log|2\psi_\lambda|\,d\mu_0\right)=|2\lambda|.
\end{equation*}

On the other hand, for any $\lambda\in(1,+\infty)$ and $\bfw_1,\bfw_2\in\Omega$, such that 

{$\bfw_2$ is obtained from $\bfw_1$ by switching all digits ``1'' to ``0'' and all digits ``0'' to ``1'',} 
the points $\psi_\lambda(\bfw_1)$ and $\psi_\lambda(\bfw_2)$ are complex conjugate. Hence, due to real symmetry, we have 
$$
\int_\Omega \arg(2\psi_\lambda)\,d\mu_0=0,
$$
which implies that the analytic map $\overline{\psi}_{\mu_0}$ is real-symmetric. The latter is possible only when $\overline{\psi}_{\mu_0}(\lambda)=2\lambda$, for all $\lambda\in\bbC\setminus\overline{\bbD}$.
\end{proof}

\begin{proof}[Proof of Theorem~\ref{main_frequency_theorem}]
Fix a compact set $K\subset\bbC\setminus\overline{\bbD}$. Without loss of generality we may assume that $K$ is the closure of an open domain compactly contained in $\bbC\setminus\overline{\bbD}$. For every $n\in\bbN$ consider a function $h_n\colon\Omega\to\bbR_{\ge 0}$ defined by the relation
$$
h_n(\bfw):= \|g_{n,\bfw}-\overline{\psi}_{\mu_0}\|_K=\|g_{n,\bfw}-2\cdot\id\|_K.
$$
(The last identity follows directly from Lemma~\ref{space_average_lemma}.)
Each function $h_n$ is $\mu_0$-measurable, since it is the supremum of countably many measurable functions $h_{n,\lambda}(\bfw)= |g_{n,\bfw}(\lambda)-\overline{\psi}_{\mu_0}(\lambda)\|$, where $\lambda$ runs over all points $(\bbQ+i\bbQ)\cap K$.

Fix the constants $\varepsilon,\delta>0$. Then, according to Egorov's Theorem and Theorem~\ref{ergodic_theorem}, there exists a subset $\Omega_\varepsilon\subset\Omega$, such that 
\begin{equation}\label{condition0_eq}
\mu_0(\Omega_\varepsilon)>1-\varepsilon/2,
\end{equation}
and $h_n\to 0$ uniformly on $\Omega_\varepsilon$.

Choose $n_0\in\bbN$ so that 
\begin{equation}\label{condition1_eq}
2^{1-n_0/2}<\frac{\varepsilon}{2}
\end{equation}
and for any $n\ge n_0$ and any $\bfw\in\Omega_\varepsilon$, we have
\begin{equation}\label{condition2_eq}
h_n(\bfw)<\frac{\delta}{2}.
\end{equation}
According to Lemma~\ref{neighborhood_lemma}, we may also assume without loss of generality that $n_0$ is sufficiently large, so that for any $n\ge n_0$ and $\bfw_1,\bfw_2\in\Omega$, we have 
\begin{equation}\label{condition3_eq}
\|g_{n,\bfw_1}-g_{n,\bfw_2}\|_K<\frac{\delta}{2},\qquad \text{whenever }d(\bfw_1,\bfw_2)\le 2^{-n}.
\end{equation}

For any $n\in\bbN$, let $\cycl_n\colon\Omega\to\Omega$ be the function defined as follows: for any $\bfw=(\omega_0,\omega_1,\ldots)\in\Omega$, the image $\cycl_n(\bfw)$ is the periodic itinerary $\cycl_n(\bfw)=(\tl\omega_0,\tl\omega_2,\ldots)\in\Omega$, such that $\tl\omega_k=\omega_{k\mod n}$, for any $k\in\bbN\cup\{0\}$, where
$$
k\mod n:=\min\{m\in\bbN\cup\{0\}\mid k-m\in n\bbZ\}.
$$ 
We note that for any $\bfw\in\Omega$, the itinerary $\cycl_n(\bfw)$ is periodic with period dividing~$n$.

For any $n\ge n_0$, consider the set $\tl\Omega_n:=\cycl_n(\Omega_\varepsilon)$. It follows from~(\ref{condition2_eq}) and~(\ref{condition3_eq}) that $\|g_\bfw-2\cdot\id\|_K<\delta$, for all $\bfw\in\tl\Omega_n$. On the other hand,~(\ref{condition0_eq}) implies that 
$$
\frac{\#\tl\Omega_n}{2^n}>1-\varepsilon/2.
$$
Now we conclude that for all $n\ge n_0$,
\begin{multline*}
\frac{\#\{\bfw\in\Omega_n\colon \|g_\bfw-2\cdot\id\|_K<\delta\}}{\#\Omega_n}>\frac{\#\tl\Omega_n-\sum_{m<n,\,m|n}\#\Omega_m}{2^n} > \\
1-\varepsilon/2-\frac{2^{1+n/2}}{2^n} >1-\varepsilon,
\end{multline*}
where the last inequality follows from~(\ref{condition1_eq}). This completes the proof of Theorem~\ref{main_frequency_theorem} {since the choice of $\varepsilon>0$ was arbitrary.}
\end{proof}

\section{Convergence of potentials outside of the Mandelbrot set}\label{Outside_M_sec}

\noindent The main purpose of this section is to give a proof of Lemma~\ref{Outside_M_convergence_lemma}. 

For every $n\in\bbN$, let $\hat{\mathcal P}_n\subset\bbC$ denote the set of all parameters $c$, for which there exists a parabolic point $(c,\mathcal O)$ on the period~$n$ curve $\PernZ$, (i.e., $\rho_n(c,\mathcal O)=1$ for some $(c,\mathcal O)\in\PernZ$). Let $P_n\colon\bbC\to\bbC$ be the polynomial defined by
$$
P_n(c):= \prod_{\tl c\in\hat{\mathcal P}_n}(c-\tl c).
$$

\begin{proposition}\label{P_n_lim_prop}
	For every $\lambda\in\bbC\setminus\overline{\bbD}$, we have
	$$
	\lim_{n\to\infty}\frac{1}{2^n}\log|P_n(\cc(\lambda))|=\log|\lambda|.
	$$
\end{proposition}
\begin{proof}
For every $n\in\bbN$, consider the function
$$
R_n(c):= \prod_{\mathcal O\mid (c,\mathcal O)\in\PernZ}(1-\rho_n(c,\mathcal O)),
$$
where the product is taken over all periodic orbits $\mathcal O$, such that $(c,\mathcal O)\in\PernZ$. According to~\cite{Bassanelli_Berteloot}, this is a polynomial, proportional to the polynomial $P_n$ with some coefficient $a_n\in\bbC$:
$$
R_n(c)=a_nP_n(c).
$$
Since $a_n$ is the leading coefficient of the polynomial $R_n$, its modulus can be estimated by
\begin{multline*}
|a_n|=\lim_{|c|\to\infty} ( |c|^{-\deg R_n}|R_n(c)| ) \\= \lim_{|c|\to\infty} \left(|c|^{-\deg R_n}\prod_{\mathcal O\mid (c,\mathcal O)\in\PernZ}|\rho_n(c,\mathcal O)|\right) 
= |c|^{-{\frac n2 (\#\Omega_n)/n}}|4 c|^{\frac n2 (\#\Omega_n)/n} \\ = 2^{\#\Omega_n}=2^{2^n}+o(2^{2^n}).
\end{multline*}

Finally, it was shown in~\cite{Buff_Gauthier} that, 
$$
\lim_{n\to\infty}\frac{1}{2^n}\log|R_n(\cc(\lambda))|=\log|\lambda|+\log 2,
$$
for any $\lambda\in\bbC\setminus\overline{\bbD}$, hence
$$
\lim_{n\to\infty}\frac{1}{2^n}\log|P_n(\cc(\lambda))|=\lim_{n\to\infty}\frac{1}{2^n}\log|a_n^{-1}R_n(\cc(\lambda))| = \log|\lambda|.
$$
\end{proof}

\begin{proposition}\label{lim_log_C_n_prop}
	For every $\lambda\in\bbC\setminus\overline{\bbD}$, we have
	$$
	\lim_{n\to\infty}\frac{1}{2^n}\log|C_n(\cc(\lambda))|=\log|\lambda|.
	$$
\end{proposition}
\begin{proof}
For every $n\ge 3$ and every $c\in\bbC$, we have $C_n(c)=P_n(c)/N_n(c)$, where
$$
N_n(c):= \prod_{\tl c\in\hat{\mathcal P}_n\setminus \tilde{\mathcal P}_n}(c-\tl c).
$$

Now we estimate the degrees of the polynomials $N_n$, for large $n$ (c.f.~(\ref{nu_phi_estimate_eq})):
$$
\deg N_n = \sum_{n=rp, p<n}\nu(p)\phi(r)\le 2^{\frac n2-1}n^2=o(2^n),
$$
where $\nu(p)$ is the same as in~(\ref{nu_n_def_eq}) and $\phi(r)$ is the number of positive integers that are smaller than $r$ and relatively prime with $r$.

Since for every $n\ge 3$, the set $\hat{\mathcal P}_n\setminus \tilde{\mathcal P}_n$ of all roots of the polynomial $N_n$ is contained in the Mandelbrot set~$\M$, and since for any $\lambda\in\bbC\setminus\overline{\bbD}$, we have $\cc(\lambda)\not\in\M$, it follows that for every $n\ge 3$ and every $\tl c\in \hat{\mathcal P}_n\setminus \tilde{\mathcal P}_n$, the inequality
$$
|\log|\cc(\lambda)-\tl c||\le K
$$
holds for some constant $K=K(\lambda)>0$.

Finally, for any $\lambda\in\bbC\setminus\overline{\bbD}$, we get
$$
\frac 1{2^n} |\log|N_n(\cc(\lambda))||\le \frac{1}{2^n}K\deg N_n=o(K)\to 0,\qquad\text{as }n\to\infty,
$$
thus, it follows from Proposition~\ref{P_n_lim_prop} that
\begin{multline*}
\lim_{n\to\infty}\frac{1}{2^n}\log|C_n(\cc(\lambda))|= \\ =\lim_{n\to\infty}\frac{1}{2^n}\left(\log|P_n(\cc(\lambda))|-\log|N_n(\cc(\lambda))|\right) =\log|\lambda|.
\end{multline*}
\end{proof}

For every simply connected domain $U\subset\bbC\setminus\M$, the double covering map $\cc$ has exactly two single-valued inverse branches defined on $U$. Let $\cc^{-1}\colon U\to\bbC\setminus\overline{\bbD}$ be any fixed inverse branch of the map $\cc$ on $U$. (It follows from~(\ref{cc_def_eq}) that the two inverses of $\cc$ differ only by a sign.) 
Now for each $n\in\bbN$ and each itinerary $\bfw\in\Omega$, we consider the maps $\tl g_{n,\bfw}, \sigma_{n,\bfw}\colon U\to\bbC$ defined by
$$
\tl g_{n,\bfw}(c):= g_{n,\bfw}(\cc^{-1}(c))\qquad\text{and}
$$
\begin{equation}\label{sigma_nw_def_eq}
\sigma_{n,\bfw}(c):= \frac{d}{dc}[(\tl g_{n,\bfw}(c))^n].
\end{equation}
In particular, if $\bfw\in\Omega$ is a periodic itinerary of period $n$, then
$$
\sigma_{n,\bfw}(c)=\sigma_n(c,\mathcal O),
$$
where $\mathcal O$ is the periodic orbit of $f_c$, containing the point $\psi_\bfw(\cc^{-1}(c))$.
\begin{remark}
We note that even though the map $\sigma_{n,\bfw}$ depends on the choice of the inverse $\cc^{-1}$, switching to a different choice of $\cc^{-1}$ in the definition of $\sigma_{n,\bfw}$ is equivalent to switching the itinerary $\bfw$ to the one where every $0$ is replaced by $1$ and every $1$ is replaced by $0$. Since all our subsequent statements will be quantified ``for every $\bfw$'', they will be independent of the choice of $\cc^{-1}$.
\end{remark}

\begin{lemma}\label{roots_linear_growth_lemma}
For every Jordan domain $U\Subset\bbC\setminus\M$, there exists a positive integer $\alpha=\alpha(U)\in\bbN$, such that for every $\bfw\in\Omega$, $s\in\bbC$ and $n\in\bbN$, the equation 
$$
\sigma_{n,\bfw}(c)=s
$$
has no more than $\alpha n$ different solutions $c\in U$, counted with multiplicities.
\end{lemma}
\begin{proof}
Assume that the statement of the lemma is false. Then there exists an increasing sequence of positive integers $\{n_k\}_{k=1}^\infty\subset\bbN$, a sequence of complex numbers $\{s_k\}_{k=1}^\infty\subset\bbC$ and the corresponding sequence of itineraries $\{\bfw_k\}_{k=1}^\infty\subset\Omega$, such that for every $k\in\bbN$, the equation
\begin{equation}\label{sigma_nk_equals_s_k_eq}
\sigma_{n_k,\bfw_k}(c)=s_k
\end{equation}
has more than $kn_k$ different solutions $c\in U$, counted with multiplicities.

Let $U_1\subset\bbC\setminus\M$ be a Jordan domain with a $C^2$-smooth boundary, such that $U\Subset U_1$. According to Proposition~\ref{normality_prop}, the family of maps $\{\tl g_{n_k,\bfw_k}\}_{k=1}^\infty$ is normal on some simply connected subdomain of $\bbC\setminus\M$ that compactly contains $U_1$, hence after extracting a subsequence, we may assume without loss of generality that the sequence of maps $\{\tl g_{n_k,\bfw_k}\}_{k=1}^\infty$ converges to a holomorphic map $\tl g\colon \overline U_1\to\bbC$ uniformly on $\overline U_1$ and the sequence of the derivatives of these maps of arbitrary order converges to the derivative of $\tl g$ of the same order uniformly on $\overline U_1$.

Let $S^1=\bbR\slash\bbZ$ be the affine circle and let $\gamma\colon S^1\to\partial U_1$ be a $C^2$-smooth parameterization of the boundary of $U_1$ in the counterclockwise direction, such that $\gamma'(t)\neq 0$ for any $t\in S^1$. For every $k\in\bbN$, let $r_k\in\bbN$ be the number of solutions of the equation~(\ref{sigma_nk_equals_s_k_eq}) in the domain $U_1$, counted with multiplicities. Then, according to the argument principle, the number of solutions $r_k$ is equal to the number of turns the curve $\sigma_{n_k,\bfw_k}\circ\gamma$ makes around the point $s_k$. (If the curve passes through the point $s_k$, then this does not count as a turn around~$s_k$.) This number of turns can be estimated from above via the total variation $TV_{0}^1\left(\arg\left[\frac{d}{dt}\sigma_{n_k,\bfw_k}(\gamma(t))\right]\right)$ of the argument of the tangent vector $\frac{d}{dt}\sigma_{n_k,\bfw_k}(\gamma(t))$, where the argument is viewed as a continuous function of $t\in[0,1]$:
\begin{multline}\label{r_k_estimate_1_eq}
r_k\le \frac{1}{2\pi} TV_{0}^1\left(\arg\left[\frac{d}{dt}\sigma_{n_k,\bfw_k}(\gamma(t))\right]\right) \\
=\frac{1}{2\pi} TV_{0}^1\left(\Im\left[\log\left(\frac{d}{dt} \sigma_{n_k,\bfw_k}(\gamma(t))\right)\right]\right) \\ \le 
\frac{1}{2\pi} TV_{0}^1\left(\Im\left[\log\left(\sigma_{n_k,\bfw_k}'(\gamma(t)) \right)\right]\right) + \frac{1}{2\pi} TV_{0}^1\left(\Im\left[\log\left(\gamma'(t)\right)\right]\right).
\end{multline}

The term $\frac{1}{2\pi} TV_{0}^1\left(\Im\left[\log\left(\gamma'(t)\right)\right]\right)$ in the right hand side of~(\ref{r_k_estimate_1_eq}) is independent of $k$, hence is a constant that depends only on the domain $U_1$.  We note that this constant is finite, since $\gamma$ is $C^2$-smooth and
$$
TV_{0}^1\left(\Im\left[\log\left(\gamma'(t)\right)\right]\right) \le  \int_0^1 \left|\frac{\gamma''(t)}{\gamma'(t)}\right|dt <\infty.
$$
Now we estimate the remaining term in the right hand side of~(\ref{r_k_estimate_1_eq}). To simplify the notation, denote 
$$
g_k(t):= \tl g_{n_k,\bfw_k}(\gamma(t)),\qquad g_k^{(1)}(t):= \tl g_{n_k,\bfw_k}'(\gamma(t)),
$$
$$
g_k^{(2)}(t):= \tl g_{n_k,\bfw_k}''(\gamma(t)),\qquad g_k^{(3)}(t):= \tl g_{n_k,\bfw_k}'''(\gamma(t)).
$$
A direct computation yields
\begin{multline*}
TV_{0}^1\left(\Im\left[\log\left(\sigma_{n_k,\bfw_k}'(\gamma(t)) \right)\right]\right)\le  \int_0^1 \left|\frac{\sigma_{n_k,\bfw_k}''(\gamma(t))}{\sigma_{n_k,\bfw_k}'(\gamma(t))}\gamma'(t)\right|dt \\
= \int_0^1 \left|G(t)\right|dt,
\end{multline*}
where
\begin{multline*}
G(t)= \left(\frac{ (n_k-1)(n_k-2)\left( g_k^{(1)}(t)\right)^3 + 3(n_k-1) g_k(t) g_k^{(1)}(t) g_k^{(2)}(t)  }{ (n_k-1)\left( g_k^{(1)}(t)\right)^2 + g_k(t)g_k^{(2)}(t) } \right.\\
\left.+\frac{ ( g_k(t))^2  g_k^{(3)}(t) }{ (n_k-1)\left( g_k^{(1)}(t)\right)^2 + g_k(t)g_k^{(2)}(t) }\right)\gamma'(t).
\end{multline*}
Without loss of generality we may assume that the derivative
$\tl g'$ does not vanish on $\partial U_1$. Otherwise we may guarantee this by shrinking the domain $U_1$ a little, so that the inclusion $U\Subset U_1$ still holds. Then 
$$
G(t) = (n_k-2)\tl g'(\gamma(t))\gamma'(t) +\frac{3\tl g(\gamma(t)) \tl g''(\gamma(t))}{\tl g'(\gamma(t))}\gamma'(t) +o(1),\qquad \text{as }k\to\infty.
$$
Hence, it follows from~(\ref{r_k_estimate_1_eq}) and the above estimates that there exist positive constants $A,B>0$, such that
$$
r_k\le A(n_k-2)+B,
$$
for all sufficiently large $k$. The latter contradicts to our original assumption that for any $k\in\bbN$, the equation~(\ref{sigma_nk_equals_s_k_eq}) has more that $kn_k$ solutions in $U$, counted with multiplicities.
\end{proof}

For a simply connected domain $U\subset\bbC\setminus\M$ and for any $n\in\bbN$, $s\in\bbC$ and $\bfw\in\Omega$, let $c_1,c_2,\dots,c_k\in U$ be all solutions of the equation $\sigma_{n,\bfw}(c)=s$ in $U$, listed with their multiplicities. Then the function $(\sigma_{n,\bfw}(c)-s)/\prod_{j=1}^k(c-c_j)$ is holomorphic as a function of $c\in U$ and has no zeros in $U$. The latter implies that
$$
f_{n,\bfw,s,U}(c):= \left(\frac{\sigma_{n,\bfw}(c)-s}{\prod_{j=1}^k(c-c_j)}\right)^{1/n}
$$
is a well-defined analytic function on $U$, for some fixed choice of the branch of the root. (We do not specify a particular choice of the branch, since further statements are independent of this choice.)

\begin{lemma}\label{normal_family_after_division_lemma}
For every Jordan domain $U\Subset\bbC\setminus\M$ and every sequence of complex numbers $\{s_n\}_{n\in\bbN}$, satisfying~(\ref{lim_sup_condition}), the family of holomorphic maps
$$
\mathcal F=\{f_{n,\bfw,s_n,U}\mid n\in\bbN,\bfw\in\Omega\}
$$
is uniformly bounded (hence, normal) in $U$. Furthermore, there exists a real number $D>0$ that depends only on the sequence $\{s_n\}_{n\in\bbN}$, such that if $\diam(U)> D$, then the identical zero map is not a limit point of the normal family $\mathcal F$.
\end{lemma}
\begin{proof}
First, we observe that for a sequence of complex numbers $\{s_n\}_{n\in\bbN}$, satisfying~(\ref{lim_sup_condition}), there exists a real number $M_1>0$, such that
$$
|s_n|\le 3^n M_1,\qquad\text{for any }n\in\bbN.
$$

As before, for any $n\in\bbN$ and $\bfw\in\Omega$, let $c_1,c_2,\dots,c_{k_{n,\bfw}}\in U$ be all solutions of the equation $\sigma_{n,\bfw}(c)=s_n$ in $U$, listed with their multiplicities. Then, for any $n\in\bbN$ and $\bfw\in\Omega$, we consider a holomorphic function $\hat f_{n,\bfw}\colon U\to\bbC$, defined by
$$
\hat f_{n,\bfw}(c):= \frac{\sigma_{n,\bfw}(c)-s_n}{\prod_{j=1}^{k_{n,\bfw}}(c-c_j)}.
$$
We note that since the function $\sigma_{n,\bfw}$ analytically extends to any simply connected domain $U_1\subset\bbC\setminus\M$, such that $U\subset U_1$, so does the function $\hat f_{n,\bfw}$. 

We fix a Jordan domain $U_1$, such that $U\Subset U_1\Subset\bbC\setminus\M$. It follows from~(\ref{sigma_nw_def_eq}) and normality of the family $\{\tl g_{n,\bfw}\mid n\in\bbN,\bfw\in\Omega\}$ in some simply connected domain compactly containing $U_1$ (c.f.~Proposition~\ref{normality_prop}) that there exists a real number $M_2>3$, such that
$$
|\sigma_{n,\bfw}(c)|\le nM_2^n,\qquad\text{for any } n\in\bbN,\bfw\in\Omega,\quad\text{and}\quad c\in\partial U_1.
$$
Let $d>0$ be the distance between the boundaries $\partial U$ and $\partial U_1$. Without loss of generality we may also assume that $d\le 1$. Then for every $n\in\bbN$, $\bfw\in\Omega$ and $c\in\partial U_1$, we have
$$
|\hat f_{n,\bfw}(c)|\le (n+M_1)M_2^nd^{-k_{n,\bfw}}\le (n+M_1)M_2^nd^{-\alpha n},
$$
where $\alpha=\alpha(U)$ is the same as in Lemma~\ref{roots_linear_growth_lemma}. By the Maximum Principle, the same inequality holds for all $c\in U$. After taking the root of degree $n$ from both sides of this inequality, we conclude that the family $\mathcal F$ is uniformly bounded on $U$, hence is normal on $U$.

In order to prove the second assertion of the lemma, we observe that if $\diam(U)> D$, then by the triangle inequality, there exists a point $c_0\in U$, such that $|c_0|\ge D/2$. 
If $D>0$ is sufficiently large, then for all $n\in\bbN$, $\bfw\in\Omega$ and for all $c$ in a neighborhood of $c_0$, we have $|\tl g_{n,\bfw}(c)|^n \sim |4c|^{n/2}$, (c.f.~\cite{Buff_Gauthier}) and hence
$$
|\sigma_{n,\bfw}(c)|= \left|\frac{d}{dc}[(\tl g_{n,\bfw}(c))^n]\right| \sim 2n|4c|^{\frac{n}{2}-1},
$$
and in particular,
$$
|\sigma_{n,\bfw}(c_0)|\ge |4c_0|^{\frac{n}{2}-1}.
$$
Then, for all $n\in\bbN$, $\bfw\in\Omega$, assuming that the constant $D$ is sufficiently large (here $D$ is required to depend only on $M_1$), we have
$$
|\hat f_{n,\bfw}(c_0)|\ge \frac{|4c_0|^{\frac{n}{2}-1}-3^nM_1}{D^{k_{n,\bfw}}}\ge \frac{(2D)^{-1}(\sqrt{2D}-3M_1)^n}{D^{\alpha n}}\ge M_3^n,
$$
for some fixed constant $M_3>0$ that does not depend on $n$ and $\bfw$. Now, after taking the root, we obtain
$$
|f_{n,\bfw,s_n,U}(c_0)|\ge M_3,
$$
for any map $f_{n,\bfw,s_n,U}\in\mathcal F$. This implies that the identical zero map is not a limit point of the normal family $\mathcal F$.
\end{proof}

\begin{lemma}\label{Outside_M_conv_partial_case_lemma}
Under conditions of Lemma~\ref{Outside_M_convergence_lemma}, let $V\Subset\bbC\setminus\M$ be a Jordan domain, such that $\diam(V)>D$, where the real number $D>0$ is the same as in Lemma~\ref{normal_family_after_division_lemma}. Then the sequence of subharmonic functions $\{u_{s_n,n}\}_{n\in\bbN}$ converges to $v=G_\M+\log 2$ in $L^1$-norm on $\overline V$, as $n\to\infty$.
\end{lemma}
\begin{proof}
Let $U\Subset\bbC\setminus\M$ be another Jordan domain, such that $V\Subset U$.
Recall that according to~(\ref{u_sn_def_eq}) and Lemma~\ref{S_n_polynomial_lemma}, for any $c\in\bbC\setminus\M$, we have
$$
u_{s_n,n}(c)=  (\deg_cS_n)^{-1}\left(\log|C_n(c)|+ \log|\tl S_n(c,s_n)|\right).
$$
Now, applying~(\ref{tilde_S_def_eq}) to the last term in the formula above and representing each term $\sigma_n(c,\mathcal O)$ as $\sigma_{n,\bfw}(c)$, for an appropriate $\bfw\in\Omega_n$, we get that for any $c\in\overline U$, the identity
$$
u_{s_n,n}(c)= (\deg_cS_n)^{-1}\left(\log|C_n(c)|+ \frac{1}{n}\sum_{\bfw\in\Omega_n}\log|\sigma_{n,\bfw}(c)-s_n|\right)
$$
holds.

We will prove the lemma by showing that for any $\varepsilon>0$, there exists $n_0\in\bbN$, such that for any $n\ge n_0$, we have
\begin{equation}\label{L1_norm_main_eq}
\|u_{s_n,n}-v\|_{L^1(\overline V)}<\varepsilon\cdot(1+\mathrm{area}(V)).
\end{equation}

Let $\cc^{-1}\colon \overline U\to\bbC\setminus\overline{\bbD}$ denote the inverse branch of the map $\cc$ chosen before Lemma~\ref{roots_linear_growth_lemma}. For any $n\in\bbN$ and $\varepsilon>0$, let $\Omega_{n,\varepsilon}\subset\Omega_n$ be the set defined by the condition that 
\begin{equation}\label{G_n_eps_def_eq}
\bfw\in\Omega_{n,\varepsilon}\qquad\text{if and only if}\qquad  \|g_\bfw-2\cdot\id\|_{\cc^{-1}(\overline U)}\le\varepsilon.
\end{equation}
Then for any $c\in\overline U$, we can represent $u_{s_n,n}(c)$ as
$$
u_{s_n,n}(c) = F_n(c)+G_{n,\varepsilon}(c)+H_{n,\varepsilon}(c),
$$
where
$$
F_n(c):= (\deg_cS_n)^{-1} \log|C_n(c)|,
$$
$$
G_{n,\varepsilon}(c):= \frac{(\deg_cS_n)^{-1}}{n} \sum_{\bfw\in \Omega_{n,\varepsilon}}\log|\sigma_{n,\bfw}(c)-s_n|,
$$
$$
H_{n,\varepsilon}(c):= \frac{(\deg_cS_n)^{-1}}{n} \sum_{\bfw\in \Omega_n\setminus \Omega_{n,\varepsilon}}\log|\sigma_{n,\bfw}(c)-s_n|.
$$

First, we observe that for any $c\in\overline V$, the identity
$$
\sigma_{n,\bfw}(c) = n(g_\bfw(\lambda))^{n-1}g_\bfw'(\lambda)\frac{d}{dc}\cc^{-1}(c),
$$
holds for $\lambda = \cc^{-1}(c)$.
Now it follows from~(\ref{G_n_eps_def_eq}) and Cauchy's estimates that for any $\bfw\in \Omega_{n,\varepsilon}$, we have $\|g_\bfw'-2\|_{\cc^{-1}(\overline V)}\le \varepsilon/r$, where $r>0$ is the distance between the boundaries $\cc^{-1}(\partial V)$ and $\cc^{-1}(\partial U)$. This implies that for all sufficiently small $\varepsilon>0$, there exists $n_1\in\bbN$, such that for all $n\ge n_1$, $\bfw\in\Omega_{n,\varepsilon}$ and any $c\in\overline V$, we get
$$
\lim_{n\to\infty}\left|\frac{s_n}{\sigma_{n,\bfw}(c)}\right|  = 0.
$$
Furthermore, as $n\to\infty$, setting $\lambda = \cc^{-1}(c)$ as before, we obtain that
\begin{multline*}
G_{n,\varepsilon}(c) = \frac{(\deg_cS_n)^{-1}}{n} \sum_{\bfw\in \Omega_{n,\varepsilon}}\left(\log|\sigma_{n,\bfw}(c)|+\log|1-s_n/\sigma_{n,\bfw}(c)|\right) \\
= \frac{1}{2^nn} \sum_{\bfw\in \Omega_{n,\varepsilon}} \left(\log n+(n-1)\log|g_\bfw(\lambda)| + \log|g_{\bfw}'(\lambda)|+\log\left|\frac{d}{dc}\cc^{-1}(c)\right|\right) +o(1) \\
= \frac{1}{2^n} \sum_{\bfw\in \Omega_{n,\varepsilon}} \log|g_\bfw(\lambda)|  +o(1),
\end{multline*}
where $o(1)$ denotes a term that converges to zero uniformly on $\overline V$, as $n\to\infty$.
According to Theorem~\ref{main_frequency_theorem}, we have an asymptotic relation $\#\Omega_{n,\varepsilon}\sim 2^n$, which together with~(\ref{G_n_eps_def_eq}) implies existence of a positive integer $n_2=n_2(\varepsilon)\in\bbN$, such that 
\begin{equation}\label{G_n_eps_lim_eq}
\left|G_{n,\varepsilon}(c)-\log|2\lambda|\right|<\varepsilon,\qquad \text{for any } n\ge n_2 \text{ and }c\in\overline V.
\end{equation}

Next, we observe that according to Lemma~\ref{roots_linear_growth_lemma} and Lemma~\ref{normal_family_after_division_lemma}, for every $n\in\bbN$ and $\bfw\in\Omega_{n}$, there exist a holomorphic function $f_{n,\bfw}\colon U\to\bbC$ and a finite number of points $c_1,c_2,\dots,c_{k_{n,\bfw}}\in U$, such that
$$
\frac 1n \log|\sigma_{n,\bfw}(c)-s_n|= \log|f_{n,\bfw}(c)|+ \frac 1n\sum_{j=1}^{k_{n,\bfw}}\log|c-c_j|.
$$
It follows from Lemma~\ref{normal_family_after_division_lemma} that there exists a constant $C_1=C_1(U,V)>1$, such that
$$
\frac{1}{C_1}<|f_{n,\bfw}(c)|<C_1,
$$
for any $n\in\bbN$, $\bfw\in\Omega_{n}$ and $c\in \overline V$. At the same time, according to Lemma~\ref{roots_linear_growth_lemma}, we have $k_{n,\bfw}\le\alpha n$, where $\alpha=\alpha(U)$ is the same as in Lemma~\ref{roots_linear_growth_lemma}. This implies that there exists a constant $C_2=C_2(U,V)>0$, such that
$$
\left\| \frac 1n \log|\sigma_{n,\bfw}(c)-s_n| \right\|_{L^1(\overline V)}<C_2.
$$
Thus, we obtain
$$
0\le \|H_{n,\varepsilon}(c)\|_{L^1(\overline V)}\le C_2(\deg_c S_n)^{-1}\cdot(\#\Omega_n-\#\Omega_{n,\varepsilon}).
$$
Now Theorem~\ref{main_frequency_theorem} and Lemma~\ref{deg_c_S_n_lim_lemma} imply that for any $\varepsilon>0$, the right hand side of the above inequality converges to zero, as $n\to\infty$, so we have
\begin{equation}\label{H_n_eps_limit_eq}
\lim_{n\to\infty} \|H_{n,\varepsilon}(c)\|_{L^1(\overline V)} =0.
\end{equation}

Finally, it follows from Proposition~\ref{lim_log_C_n_prop} and Lemma~\ref{deg_c_S_n_lim_lemma} that for any $c\in \overline V$, we have
$$
\lim_{n\to\infty} F_n(c) = \log|\lambda|,
$$
where $\lambda = \cc^{-1}(c)$. Together with~(\ref{G_n_eps_lim_eq}),~(\ref{H_n_eps_limit_eq}) and~(\ref{log_lambda_Green_eq}), this implies~(\ref{L1_norm_main_eq}), for an arbitrary $\varepsilon>0$ and all sufficiently large $n\in\bbN$. This completes the proof of the lemma.
\end{proof}

\begin{proof}[Proof of Lemma~\ref{Outside_M_convergence_lemma}]
Given a compact set $K\subset\bbC\setminus\M$, there exist two Jordan domains $V_1,V_2\Subset\bbC\setminus\M$, satisfying the conditions of Lemma~\ref{Outside_M_conv_partial_case_lemma}, and such that $K\subset\overline V_1\cup\overline V_2$. Then applying Lemma~\ref{Outside_M_conv_partial_case_lemma} to both $V_1$ and $V_2$, we conclude that
$$
\lim_{n\to\infty}\|u_{s_n,n} -G_\M-\log 2\|_{L^1(K)} = 0,
$$
which completes the proof of the lemma.
\end{proof}

\bibliographystyle{amsalpha}
\bibliography{biblio}
\end{document}